\documentclass[12pt]{amsart}
\usepackage{amsmath,amssymb,txfonts}
      \newtheorem{theorem}{Theorem}[section]
      \newtheorem{remark}[theorem]{Remark}

      \newcommand{\ct}[1]{\langle {#1}\rangle \lower.3ex\hbox{$_{t}$}}
      \newcommand{\lt}[1]{[ {#1}] \lower.3ex\hbox{$_{t}$}}

\begin{document}

\title[$p$-capacity vs surface-area]{$p$-capacity vs surface-area}

\author{J. Xiao}
%    Address of record for the research reported here
\address{Department of Mathematics and Statistics, Memorial University, St. John's, NL A1C 5S7, Canada}
\email{jxiao@mun.ca}
\thanks{This project was in part supported by MUN's University Research Professorship (208227463102000) and NSERC of Canada.}

%    General info

\date{}

%\dedicatory{In memory of ...}


\begin{abstract}
This paper is devoted to exploring the relationship between the $[1,n)\ni p$-capacity and the surface-area in $\mathbb R^{n\ge 2}$ which especially shows: if $\Omega\subset\mathbb R^n$ is a convex, compact, smooth set with its interior $\Omega^\circ\not=\emptyset$ and the mean curvature $H(\partial\Omega,\cdot)>0$ of its boundary $\partial\Omega$ then
$$
\left(\frac{n(p-1)}{p(n-1)}\right)^{p-1}\le\frac{\left(\frac{\hbox{cap}_p(\Omega)}{\big(\frac{p-1}{n-p}\big)^{1-p}\sigma_{n-1}}\right)}{\left(\frac{\hbox{area}(\partial\Omega)}{\sigma_{n-1}}\right)^\frac{n-p}{n-1}}\le\left(\sqrt[n-1]{\int_{\partial\Omega}\big(H(\partial\Omega,\cdot)\big)^{n-1}\frac{d\sigma(\cdot)}{\sigma_{n-1}}}\right)^{p-1}\quad\forall\quad p\in (1,n)
$$
whose limits $1\leftarrow p\ \&\ p\rightarrow n$ imply
$$
1=\frac{cap_1(\Omega)}{\hbox{area}(\partial\Omega)}\ \ \& \  \int_{\partial\Omega}\big(H(\partial\Omega,\cdot)\big)^{n-1}\frac{d\sigma(\cdot)}{\sigma_{n-1}}\ge 1,
$$
thereby not only discovering that the new best known constant is roughly half as far from the one conjectured by P\'olya-Szeg\"o in \cite[(2)]{P} but also extending the P\'olya-Szeg\"o inequality in \cite[(5)]{P}, with both the conjecture and the inequality being stated for the electrostatic capacity of a convex solid in $\mathbb R^3$.  
\end{abstract}
\maketitle

%\tableofcontents
\section{Overview}\label{s0}
\setcounter{equation}{0}

Given a compact set $\Omega$ in the $2\le n$-dimensional Euclidean space $\mathbb R^n$ equipped with the standard volume and surface-area elements $d\nu$ and $d\sigma$. The variational $[1,n)\ni p$-capacity of $\Omega$ is defined by
$$
\hbox{cap}_p(\Omega)=\inf\left\{\int_{\mathbb R^n}|\nabla f|^p\,d\nu:\ \ f\in C_c^\infty(\mathbb R^n)\ \ \&\ \ f(x)\ge 1\ \ \forall\ \ x\in\Omega\right\},
$$
where $C^\infty_c(\mathbb R^n)$ is the class of all infinitely differentiable functions with compact support in $\mathbb R^n$. Equivalently, the above infimum can be taken over either all $f\in C^\infty_c(\mathbb R^n)$ with $f=1$ in a neighbourhood of $\Omega$, or all Lipschitz functions $u$ on $\mathbb R^n$ with $f=1$ in a neighbourhood of $\Omega$ (cf. \cite[pp. 27-28]{HKM}).

As a set function on compact subsets of $\mathbb R^n$, $\hbox{cap}_p(\cdot)$ enjoys the following basic properties (a) through (f) (cf. \cite[pp. 28-32]{HKM} and \cite[Lemma 2.2.5]{Maz}):

\begin{itemize}

\item[(a)] Boundarization -- if $\Omega$ is a compact subset of $\mathbb R^n$ with non-empty boundary $\partial\Omega$ then $$\hbox{cap}_p(\partial\Omega)=\hbox{cap}_p(\Omega).$$

\item[(b)] Monotonicity -- if $\Omega_1$ and $\Omega_2$ are compact subsets of $\mathbb R^n$ with $\Omega_1\subseteq\Omega_2$ then $$\hbox{cap}_p(\Omega_1)\le \hbox{cap}_p(\Omega_2).$$

\item[(c)] Continuity -- if $(\Omega_j)_{j=1}^\infty$ is a decreasing sequence of compact subsets of $\mathbb R^n$ then $$\hbox{cap}_p(\cap_{j=1}^\infty \Omega_j)=\lim_{j\to\infty}\hbox{cap}_p(\Omega_j).$$

\item[(d)] Ball capacity -- if $B(x,r)=\{y\in\mathbb R^n: |y-x|\le r\}$ and $\sigma_{n-1}$ is the surface area of the origin-centred unit ball $B(0,1)$ then $$\hbox{cap}_p\big(B(x,r)\big)=r^{n-p}\Big(\frac{p-1}{n-p}\Big)^{1-p}\sigma_{n-1}.$$

\item[(e)] Geometric endpoint -- if $\Omega$ is a compact subset of $\mathbb R^n$ and $\hbox{area}(\cdot)$ stands for the surface-area of a set in $\mathbb R^n$ then $$\hbox{cap}_1(\Omega)=\inf\big\{\hbox{area}(\partial\Lambda):\ \Omega\subset\Lambda\cup\partial \Lambda\ \hbox{with\ bound\ open}\ \Lambda\ \hbox{and\ smooth}\ \partial\Lambda\big\}.$$ 

\item[(f)] Physical interpretation -- if $\Omega$ is a compact subset of $\mathbb R^{n\ge 3}$, then $\hbox{cap}_2(\Omega)$ is the maximal charge which can be placed on $\Omega$ when the electrical potential of the vector field created by this charge is controlled by $1$, namely, 
$$
\hbox{cap}_2(\Omega)=\sup\Big\{\mu(\Omega):\ \hbox{measure}\ \mu\ \hbox{with}\ \hbox{supp}(\mu)\subseteq\Omega\ \&\  \int_{\mathbb R^n}|x-y|^{2-n}\,\frac{d\mu(y)}{(n-2)\sigma_{n-1}}\le 1\ \forall\ x\in\mathbb R^n\setminus \Omega\Big\}.
$$
\end{itemize}

Motivated by P\'olya's 1947 paper \cite{P} as well as (a)\&(e) above, this article stems from discovering the relationship between the $p$-capacity and the surface-area (via the mean curvature). The details for such a discovery are provided in \S\ref{s2}\&\S \ref{s3} whose summary is shown in the sequel:
 
 \begin{itemize}
 \item[(h)] Surface area to variational capacity (\S\ref{s2}) -- In Theorem \ref{t11} we use the convexity of level set of $(1,n)\ni p$-equilibrium potential and a minimizing technique to gain \eqref{eAV}, a sharp convexity type inequality, linking the normalized variational capacity, the normalized surface area and the normalized volume and consequently deriving that $\big(\frac{n(p-1)}{p(n-1)}\big)^{p-1}$ times $\big(\frac{n-p}{n-1}\big)$-th power of the normalized surface area is the asymptotically sharp lower bound of the normalized variational capacity, whence having half-solved \footnote{Namely, the new best known constant is roughly half as far from the conjectured one.} the P\'olya-Szeg\"o conjecture (for $\hbox{cap}_2(\cdot)$ in $\mathbb R^3$) that {\it of all convex bodies, with a given surface area, the circular disk has the minimum capacity.}; 
 
 \item[(i)] Variational capacity to surface area (\S\ref{s3}) -- In Theorem \ref{t21} we employ a level set formulation of the inverse mean curvature flow (generated by a kind of $1$-equilibrium potential) to achieve \eqref{e29}, a log-convexity type inequality involving the normalized variational capacity, the normalized surface area and the normalized Willmore functional for the mean curvature and consequently revealing that the product of both $\big(\frac{p-1}{n-1}\big)$-th power of the normalized Willmore functional for the mean curvature and $\big(\frac{n-p}{n-1}\big)$-th power of the normalized surface area is the optimal upper bound of the normalized variational capacity, thereby extending the P\'olya-Szeg\"o principle (for $\hbox{cap}_2(\cdot)$ in $\mathbb R^3$) that {\it unless the convex solid is a ball the capacity is less than the mean-curvature-radius}.
\end{itemize}

Naturally, a combination of \eqref{e14} in Theorem \ref{t11} and \eqref{e29e} in Theorem \ref{t21} derives that if $\Omega\subset\mathbb R^n$ is a convex, compact, smooth set with its interior $\Omega^c\not=\emptyset$ and the mean curvature $H(\partial\Omega,\cdot)>0$ of its boundary $\partial\Omega$ then
\begin{itemize}

\item[(j)] 
$$
\left(\frac{n(p-1)}{p(n-1)}\right)^{p-1}\le\frac{\left(\frac{\hbox{cap}_p(\Omega)}{\big(\frac{p-1}{n-p}\big)^{1-p}\sigma_{n-1}}\right)}{\left(\frac{\hbox{area}(\partial\Omega)}{\sigma_{n-1}}\right)^\frac{n-p}{n-1}}\le\left(\sqrt[n-1]{\int_{\partial\Omega}\big(H(\partial\Omega,\cdot)\big)^{n-1}\frac{d\sigma(\cdot)}{\sigma_{n-1}}}\right)^{p-1}\quad\forall\quad p\in (1,n)
$$
whose limiting cases  $1\leftarrow p\ \&\ p\rightarrow n$ surprisingly yield the extremal case of (e) (cf. \cite{LXZ}) and the Willmore inequality (cf. \cite{chen, RS, BM}) as seen below:
\item[(k)]
$$
1=\frac{cap_1(\Omega)}{\hbox{area}(\partial\Omega)}\ \ \& \ \int_{\partial\Omega}\big(H(\partial\Omega,\cdot)\big)^{n-1}\frac{d\sigma(\cdot)}{\sigma_{n-1}}\ge 1.
$$
\end{itemize}

\section{Surface-area to $p$-capacity}\label{s2}
\setcounter{equation}{0}

In \cite[p.12]{PS} (cf. \cite{P}) P\'olya-Szeg\"o conjectured that for any convex compact subset $\Omega$ of $\mathbb R^3$ one has

\begin{equation}
\label{e11}
{\hbox{cap}_2(\Omega)}\ge \left(4\sqrt{\frac2\pi}\right){\sqrt{\hbox{area}(\partial\Omega)}}
\end{equation}
with equality if and only if $\Omega$ is a two-dimensional disk in $\mathbb R^3$. Here it is perhaps worth pointing out that if $\Omega\subset\mathbb R^2$ then $\hbox{area}(\partial\Omega)$ is replaced by two times of the two-dimensional Lebesgue measure of $\Omega$.

The first remarkable result approaching the conjecture was obtained in P\'olya-Szeg\"o's 1951 monograph: \cite[p.165,(4)]{PS} (as a sequel to the work presented in their
1945 paper \cite{PSajm}) via suitable symmetrization and projection for any given convex compact set $\Omega\subset\mathbb R^3$: 

\begin{equation}
\label{e12}
{\hbox{cap}_2(\Omega)}\ge\left(\frac{4}{\sqrt{\pi}}\right){\sqrt{\hbox{area}(\partial\Omega)}}.
\end{equation}
Since then, no improvement has been made on (\ref{e12}) and of course (\ref{e11}) has not yet been verified - see \cite{La, CFG, FGP, HPR} for an up-to-date report on this research. In the sequel, with the help of the isocapacitary inequality for the volume $\hbox{vol}(\cdot{})$ of a level set of the equilibrium potential of an arbitrary convex compact set $\Omega\subset\mathbb R^3$ we show 
\begin{equation}
\label{e13}
{\hbox{cap}_2(\Omega)}\ge\left(\frac{3\sqrt{\pi}}{2}\right){\sqrt{\hbox{area}(\partial\Omega)}},
\end{equation}
whence finding that (\ref{e13}) holds the nearly middle place between (\ref{e11}) and (\ref{e12}) in the sense of
$$
\begin{cases}
 4\sqrt{\frac2\pi}>\frac{3\sqrt{\pi}}{2}> \frac{4}{\sqrt{\pi}};\\
4\sqrt{\frac2\pi}-\frac{3\sqrt{\pi}}{2}=0.532857...;\\
\frac{3\sqrt{\pi}}{2}-\frac{4}{\sqrt{\pi}}=0.401922....
\end{cases}
$$

As a matter of fact, we discover the brand-new sharp convexity type inequality \eqref{eAV} (for the surface-area, the variational capacity and the volume) whose by-product (\ref{e14}) is much more general than (\ref{e13}).

\begin{theorem}\label{t11} Let $\Omega$ be a convex compact subset of $\mathbb R^n$ with $\hbox{area}(\partial\Omega)>0$. Then 
\begin{equation}
\label{eAV}
\frac{n(p-1)}{p(n-1)}\left(\frac{\Big(\frac{\hbox{area}(\partial\Omega)}{\sigma_{n-1}}\Big)^\frac{1}{n-1}}{\left(\frac{\hbox{cap}_p(\Omega{})}{\big(\frac{p-1}{n-p}\big)^{1-p}\sigma_{n-1}}\right)^\frac{1}{n-p}}\right)^\frac{n-p}{p-1}+\frac{n-p}{p(n-1)}\left(\frac{\Big(\frac{\hbox{vol}(\Omega{})}{n^{-1}\sigma_{n-1}}\Big)^\frac1n}{\left(\frac{\hbox{area}(\partial\Omega)}{\sigma_{n-1}}\right)^\frac{1}{n-1}}\right)^n\le1\ \ \forall\ \ p\in (1,n)
\end{equation}
holds with equality if and only if $\Omega$ is a ball. Consequently
\begin{equation}
\label{e14}
\left(\frac{\hbox{area}(\partial\Omega)}{\sigma_{n-1}}\right)^\frac{n-p}{n-1}\le\left(\frac{\hbox{cap}_p(\Omega)}{\Big(\frac{p-1}{n-p}\Big)^{1-p}\sigma_{n-1}}\right)\left(\frac{p(n-1)}{n(p-1)}\right)^{p-1}\ \forall\ p\in (1,n),
\end{equation}
which is asymptotically optimal in the sense that if $p\to 1$ or $p\to n$ in \eqref{e14} then
\begin{equation}
\label{e15}
\hbox{area}(\partial\Omega)=\hbox{cap}_1(\Omega{})\quad\hbox{or}\quad 1=1.
\end{equation}
\end{theorem}

\begin{proof} First of all, since $\hbox{area}(\partial\Omega)>0$ and $\Omega$ is convex, it follows from \cite{LXZ} that $\hbox{cap}_1(\Omega)=\hbox{area}(\partial\Omega)>0$. In accordance with \cite[Theorem 3.2]{Xu}, if $1\le p_1<p_2<n$ then there is a constant $c(p_1,p_2,n)>0$ depending only on $(p_1,p_2,n)$ such that 
$$
\big(\hbox{cap}_{p_1}(\Omega)\big)^\frac{1}{n-p_1}\le c(p_1,p_2,n)\big(\hbox{cap}_{p_2}(\Omega)\big)^\frac{1}{n-p_2}.
$$
Upon choosing $p_1=1<p_2=p<n$, one gets $\hbox{cap}_p(\Omega)>0$.

Next, we verify \eqref{eAV} through considering two situations.

{\it Situation 1}: suppose that the interior $\Omega^\circ$ of $\Omega$ is not empty and the boundary $\partial\Omega$ of $\Omega$ is of $C^1$-smoothness. In accordance with \cite{CS,Lewis}, there is a unique $(1,n)\ni p$-equilibrium potential $u$ of $\Omega$ (not only smooth in $\Omega^c=\mathbb R^n\setminus\Omega$ but also continuous in $\mathbb R^n\setminus\Omega^\circ$) such that:
\begin{itemize}
\item
$\hbox{div}(|\nabla u|^{p-2}\nabla u)=0$ {in} $\Omega^c;$

\item $u|_{\partial\Omega}=1;$

\item $\lim_{|x|\to\infty}u(x)=0;$

\item $0<u<1$ in $\Omega^c$;

\item $|\nabla u|\not=0$ in ${\Omega}^c$;

\item 
$$
\hbox{cap}_p(\Omega)=\int_{\mathbb R^n\setminus{\Omega}}|\nabla u|^p\,d\nu =\int_{\{x\in\mathbb R^n:\ u(x)=t\}}|\nabla u|^{p-1}\,d\sigma\quad\forall\quad t\in (0,1);
$$

\item if $u$ is set to be $1$ on $\Omega$ then $\{x\in\mathbb R^n:\ u(x)\ge t\}$ is convex and $\{x\in\mathbb R^n:\ u(x)=t\}$ is smooth for any $t\in (0,1)$.
\end{itemize}
Consequently, we can utilize the well-known monotonicity for the area function of convex domains, the H\"older inequality and the co-area formula to get 
\begin{align*}
&\hbox{area}(\partial\Omega)\\
&\le\hbox{area}(\{x\in\mathbb R^n:\ u(x)=t\})\\
&=\int_{\{x\in\mathbb R^n:\ u(x)=t\}}\,d\sigma\\
&\le\left(\int_{\{x\in\mathbb R^n:\ u(x)=t\}}|\nabla u|^{p-1}\,d\sigma\right)^\frac1p\left(\int_{\{x\in\mathbb R^n:\ u(x)=t\}}|\nabla u|^{-1}\,d\sigma\right)^\frac{p-1}p\\
&=\big(\hbox{cap}_p(\Omega)\big)^\frac1p \left(-\frac{d}{dt}\hbox{vol}\big(\{x\in\mathbb R^n:\ u(x)\ge t\}\big)\right)^\frac{p-1}{p},
\end{align*}
and accordingly,
\begin{equation}
\label{sim}
\left(\frac{\hbox{area}(\partial\Omega)}{\big(\hbox{cap}_p(\Omega)\big)^\frac1p}\right)^\frac{p}{p-1}\le -\frac{d}{dt}\hbox{vol}\big(\{x\in\mathbb R^n:\ u(x)\ge t\}\big),
\end{equation}
where 
$$
\hbox{vol}\big(\{x\in\mathbb R^n:\ u(x)\ge t\}\big)
$$
is the Lebesgue measure of the upper level set $\{x\in\mathbb R^n:\ u(x)\ge t\}$.
Recalling the Poincar\'e-Mazya isocapacitary inequality (cf. \cite{PS} for $p=2$ and \cite{Maz} for $p\in (1,n)$)
$$
\frac{\hbox{vol}\big(\{x\in\mathbb R^n:\ u(x)\ge t\}{}\big)}{n^{-1}{\sigma_{n-1}}}\le \left(\frac{\hbox{cap}_p(\{x\in\mathbb R^n:\ u(x)\ge t\}{})}{\Big(\frac{p-1}{n-p}\Big)^{1-p}\sigma_{n-1}}\right)^\frac{n}{n-p}
$$
and using (a) - the boundarization of $\hbox{cap}_p(\cdot)$ to achieve the following formula (cf. \cite{PS, PP} for $p=2$)
\begin{align*}
&\hbox{cap}_p(\{x\in\mathbb R^n:\ u(x)\ge t\})\\
&=\hbox{cap}_p(\{x\in\mathbb R^n:\ u(x)=t\})\\
&=\int_{\{x\in\mathbb R^n:\ u(x)=t\}}\Big({t^{-1}|\nabla u|}\Big)^{p-1}\,d\sigma\\ 
&=t^{1-p}\hbox{cap}_p(\Omega),
\end{align*}
we obtain via integrating both sides of \eqref{sim} over the interval $(t,1)$
\begin{align*}
&(1-t)\left(\frac{\hbox{area}(\partial\Omega)}{\Big(\hbox{cap}_p(\Omega)\Big)^\frac1p}\right)^\frac{p}{p-1}\\
&\le\hbox{vol}\big(\{x\in\mathbb R^n:\ u(x)\ge t\}\big)-\hbox{vol}\big(\Omega)\\
&\le\Big(\frac{\sigma_{n-1}}{n}\Big)\left(\frac{\hbox{cap}_p(\{x\in\mathbb R^n:\ u(x)\ge t\})}{\Big(\frac{p-1}{n-p}\Big)^{1-p}\sigma_{n-1}}\right)^\frac{n}{n-p}-\hbox{vol}\big(\Omega)\\
&=\Big(\frac{\sigma_{n-1}}{n}\Big)\left(\frac{t^{1-p}\hbox{cap}_p(\Omega{})}{\Big(\frac{p-1}{n-p}\Big)^{1-p}\sigma_{n-1}}\right)^\frac{n}{n-p}-\hbox{vol}(\Omega).
\end{align*}
Note that the above estimate is valid for any $t\in [0,1]$. But if 
$$
 t\in\left[1,\ \left(\frac{\Big(\frac{\hbox{vol}(\Omega)}{n^{-1}\sigma_{n-1}}\Big)^\frac1n}{\left(\frac{\hbox{cap}_p(\Omega{})}{\big(\frac{p-1}{n-p}\big)^{1-p}\sigma_{n-1}}\right)^\frac{1}{n-p}}\right)^\frac{n-p}{1-p}\right]
$$
then 
$$
(1-t)\left(\frac{\hbox{area}(\partial\Omega)}{\Big(\hbox{cap}_p(\Omega)\Big)^\frac1p}\right)^\frac{p}{p-1}\le 0\le \Big(\frac{\sigma_{n-1}}{n}\Big)\left(\frac{t^{1-p}\hbox{cap}_p(\Omega{})}{\Big(\frac{p-1}{n-p}\Big)^{1-p}\sigma_{n-1}}\right)^\frac{n}{n-p}-\hbox{vol}(\Omega)
$$
and hence one has:
$$
(1-t)\left(\frac{\hbox{area}(\partial\Omega)}{\Big(\hbox{cap}_p(\Omega)\Big)^\frac1p}\right)^\frac{p}{p-1}\le\Big(\frac{\sigma_{n-1}}{n}\Big)\left(\frac{t^{1-p}\hbox{cap}_p(\Omega{})}{\Big(\frac{p-1}{n-p}\Big)^{1-p}\sigma_{n-1}}\right)^\frac{n}{n-p}-\hbox{vol}(\Omega)\ \ \ \forall\ \ \ t\in\left[0,\ \left(\frac{\Big(\frac{\hbox{vol}(\Omega)}{n^{-1}\sigma_{n-1}}\Big)^\frac1n}{\left(\frac{\hbox{cap}_p(\Omega{})}{\big(\frac{p-1}{n-p}\big)^{1-p}\sigma_{n-1}}\right)^\frac{1}{n-p}}\right)^\frac{n-p}{1-p}\right].
$$
Suppose $t_0$ is the critical point of the following function
$$
t\mapsto\phi(t)= (1-t)\left(\frac{\hbox{area}(\partial\Omega)}{\Big(\hbox{cap}_p(\Omega)\Big)^\frac1p}\right)^\frac{p}{p-1}-\Big(\frac{\sigma_{n-1}}{n}\Big)\left(\frac{t^{1-p}\hbox{cap}_p(\Omega{})}{\Big(\frac{p-1}{n-p}\Big)^{1-p}\sigma_{n-1}}\right)^\frac{n}{n-p}+\hbox{vol}(\Omega).
$$
Then solving $\phi'(t_0)=0$ and using the classical isoperimetric inequality one gets
$$
t_0=\left(\frac{\Big(\frac{\hbox{area}(\partial\Omega)}{\sigma_{n-1}}\Big)^\frac{1}{n-1}}{\left(\frac{\hbox{cap}_p(\Omega{})}{\big(\frac{p-1}{n-p}\big)^{1-p}\sigma_{n-1}}\right)^\frac{1}{n-p}}\right)^\frac{n-p}{1-p}\le \left(\frac{\Big(\frac{\hbox{vol}(\Omega)}{n^{-1}\sigma_{n-1}}\Big)^\frac1n}{\left(\frac{\hbox{cap}_p(\Omega{})}{\big(\frac{p-1}{n-p}\big)^{1-p}\sigma_{n-1}}\right)^\frac{1}{n-p}}\right)^\frac{n-p}{1-p},
$$
whence deriving
$$
(1-t_0)\left(\frac{\hbox{area}(\partial\Omega)}{\Big(\hbox{cap}_p(\Omega)\Big)^\frac1p}\right)^\frac{p}{p-1}\le\Big(\frac{\sigma_{n-1}}{n}\Big)\left(\frac{t_0^{1-p}\hbox{cap}_p(\Omega{})}{\Big(\frac{p-1}{n-p}\Big)^{1-p}\sigma_{n-1}}\right)^\frac{n}{n-p}-\hbox{vol}(\Omega),
$$
which implies 
$$
\frac{\hbox{vol}(\Omega)}{n^{-1}\sigma_{n-1}}\le\left(\frac{\hbox{area}(\partial\Omega)}{\sigma_{n-1}}\right)^\frac{n}{n-1}-\left(\frac{1-t_0}{n^{-1}\sigma_{n-1}}\right)\left(\frac{\hbox{area}(\partial\Omega)}{\big(\hbox{cap}_p(\Omega)\big)^\frac1p}\right)^\frac{p}{p-1},
$$
namely,
$$
1-t_0\le \Big(\frac{n-p}{n(p-1)}\Big)t_0\left(1-\frac{\Big(\frac{\hbox{vol}(\Omega{})}{n^{-1}\sigma_{n-1}}\Big)}{\left(\frac{\hbox{area}(\partial\Omega)}{\sigma_{n-1}}\right)^\frac{n}{n-1}}\right),
$$
and then \eqref{eAV} via a further computation with $t_0$. 

{\it Situation 2}: suppose that $\Omega$ is a general convex compact subset of $\mathbb R^n$. For this setting there is a sequence of convex compact sets $(\Omega_j)_{j=1}^\infty$ such that $\Omega_j^\circ\not=\emptyset$, $\partial\Omega_j$ is of $C^1$-smoothness, and $\Omega_j$ decreases to $\Omega$. Since (\ref{eAV}) and (\ref{e14}) are valid for $\Omega_j$, an application of the continuity for $\hbox{area}(\cdot{})$, $\hbox{vol}(\cdot{})$, and $\hbox{cap}_p(\cdot{})$ acting on convex compact sets ensures that \eqref{eAV} is true for such $\Omega$. 

After that, we check the equality case of \eqref{eAV}. If $\Omega$ is a ball, then an application of both (d) and the identity
$$
\frac{n(p-1)}{p(n-1)}+\frac{n-p}{p(n-1)}=1
$$
makes equality of \eqref{eAV} happen. Conversely, if equality of \eqref{eAV} occurs for all $p\in (1,n)$, then
$$
\frac{n(p-1)}{p(n-1)}\left(\frac{\Big(\frac{\hbox{area}(\partial\Omega)}{\sigma_{n-1}}\Big)^\frac{1}{n-1}}{\left(\frac{\hbox{cap}_p(\Omega{})}{\big(\frac{p-1}{n-p}\big)^{1-p}\sigma_{n-1}}\right)^\frac{1}{n-p}}\right)^\frac{n-p}{p-1}+\frac{n-p}{p(n-1)}\left(\frac{\Big(\frac{\hbox{vol}(\Omega{})}{n^{-1}\sigma_{n-1}}\Big)^\frac1n}{\left(\frac{\hbox{area}(\partial\Omega)}{\sigma_{n-1}}\right)^\frac{1}{n-1}}\right)^n=1\ \ \forall\ \ p\in (1,n).
$$
Upon letting $p\to 1$ in this last equality and using the known fact that (cf. \cite{Mey, LXZ})
$$
\liminf_{p\to 1}\hbox{cap}_p(\Omega)=\hbox{cap}_1(\Omega)=\hbox{area}(\partial\Omega)$$
we obtain
$$
{\left(\frac{\hbox{vol}(\Omega{})}{n^{-1}\sigma_{n-1}}\right)^\frac1n}={\left(\frac{\hbox{area}(\partial\Omega)}{\sigma_{n-1}}\right)^\frac{1}{n-1}},
$$
namely, equality of the isoperimetric inequality holds for $\Omega$, thereby finding that $\Omega$ is a ball.

Finally, let us deal with \eqref{e14} and its limiting cases. Note that the second term of the left-hand-side of \eqref{eAV} is non-negative. So, \eqref{e14} follows immediately from \eqref{eAV}. Moreover, the first identity of (\ref{e15}), as the limit case $p\to 1$ of (\ref{e14}), is well-known; see also \cite{LXZ}, \cite{GH} and \cite[Lemma 2.2.5]{Maz}. To see the second identity of \eqref{e15}, let $B(0,R_0)$ be an origin-symmetric ball containing $\Omega$. Using \eqref{e14} and (b)\&(d) we find
$$
1=\liminf_{p\to n}\left(\frac{\hbox{area}(\partial\Omega)}{\sigma_{n-1}}\right)^\frac{n-p}{n-1}\le\liminf_{p\to n}\left(\frac{\hbox{cap}_p(\Omega)}{\Big(\frac{p-1}{n-p}\Big)^{1-p}\sigma_{n-1}}\right)\left(\frac{p(n-1)}{n(p-1)}\right)^{p-1}\le\liminf_{p\to n}R_0^{n-p}=1,
$$
as desired.
\end{proof}

\begin{remark}
\label{r23} Below are two comments on (\ref{e14}) of independent interest:

{\rm (i)} In accordance with \cite[Proposition 1.1]{J}, if $\Omega$ is a convex compact subset of $\mathbb R^{n\ge 3}$ with $\Omega^\circ\not=\emptyset$ and smooth $\partial\Omega$, and  $u$ is the $p=2$-equllibrium potential of $\Omega$, then an application of the fact that
$$
x\mapsto v(x)=\int_{\partial\Omega}|x-y|^{2-n}\,\frac{d\sigma(y)}{(n-2)\sigma_{n-1}}
$$
is harmonic in $\mathbb R^n\setminus\partial\Omega$ (cf. \cite{MR}) gives
\begin{equation*}
v(x)=v_\infty \big((n-2)\sigma_{n-1})\big)^{-1}|x|^{2-n}+\mathcal{O}(|x|^{1-n})\quad\hbox{as}\quad |x|\to\infty,
\end{equation*}
where
$$
v_\infty=\int_{\partial\Omega}v|\nabla u|\,d\sigma.
$$
Note that (cf. \cite{MR})
$$
v(x)=\big((n-2)\sigma_{n-1})\big)^{-1}\hbox{area}(\partial\Omega{})|x|^{2-n}+\mathcal{O}(|x|^{1-n})\quad\hbox{as}\quad |x|\to\infty.
$$
So, one has
\begin{align}\label{infty}
&\big((n-2)\sigma_{n-1})\big)^{-1}\hbox{area}(\partial\Omega)\nonumber\\
&=v_\infty\nonumber\\
&=\int_{\partial\Omega}v|\nabla u|\,d\sigma\\
&\le\big(\max_{x\in\partial\Omega}v(x)\big)\int_{\partial\Omega}|\nabla u|\,d\sigma\nonumber\\
&=\big(\max_{x\in\partial\Omega}v(x)\big)\hbox{cap}_2(\Omega{}).\nonumber
\end{align}
Using the well-known layer-cake formula under $d\sigma$, one finds
\begin{align*}
&v(x)\Big((n-2)\sigma_{n-1})\Big)\\
&=\int_0^\infty\sigma\big(\{y\in\partial\Omega:\ |x-y|^{2-n}\ge t\}{}\big)\,dt\\
&=\left(\int_0^r+\int_r^\infty\right)\sigma{\small }\big(\{y\in\partial\Omega:\ |x-y|^{2-n}\ge t\}{}\big)\,dt\\
&\le \hbox{area}\big(\partial\Omega{}\big)r+(n-2)\sigma_{n-1} r^\frac1{2-n}.
\end{align*}
Minimizing the last quantity, one gets that
$$
r=\left(\frac{\hbox{area}(\partial\Omega)}{\sigma_{n-1}}\right)^\frac{2-n}{n-1}
$$
derives
\begin{equation}
\label{e23}
\int_{\partial\Omega}|x-y|^{2-n}\,\frac{d\sigma(y)}{\sigma_{n-1}}\le (n-1)\left(\frac{\hbox{area}(\partial\Omega)}{\sigma_{n-1}}\right)^\frac{1}{n-1}.
\end{equation}
This \eqref{e23}, along with \eqref{infty}, yields

\begin{equation}
\label{e24}
\left(\frac{\hbox{area}(\partial\Omega)}{\sigma_{n-1}}\right)^\frac{n-2}{n-1}\le(n-1)\left(\frac{\hbox{cap}_2(\Omega{})}{(n-2)\sigma_{n-1}}\right).
\end{equation}
The inequality (\ref{e24}) is weaker than the case $p=2$ of (\ref{e14}). However, (\ref{e24}) can be strengthened upon demonstrating the following conjecture
\begin{equation}
\label{e25}
\int_{\partial\Omega}|x-y|^{2-n}\,\frac{d\sigma(y)}{\sigma_{n-1}}\le \left(\frac{\hbox{area}(\partial\Omega)}{\sigma_{n-1}}\right)^\frac{1}{n-1}\quad\forall\quad x\in\partial\Omega,
\end{equation}
with equality if and only if $\Omega$ is a ball; see \cite[p.249,(4)]{LL}, \cite{MR} and \cite{GK} for some information related to \eqref{e25}.

{\rm (ii)} The higher dimensional extension of the variational principle presented in \cite[Theorem 1.1]{Ra} derives that if $\Omega$ is a convex compact subset of $\mathbb R^n$ with $\Omega^\circ\not=\emptyset$ and smooth $\partial\Omega$ then
\begin{equation}\label{e26}
\frac{(n-2)\sigma_{n-1}}{\hbox{cap}_2(\Omega{})}\le\frac{\int_{\partial\Omega}\int_{\partial\Omega}|x-y|^{2-n}\,d\sigma(x)d\sigma(y)}{\big(\hbox{area}(\partial\Omega)\big)^2}.
\end{equation}
A combination of (\ref{e23}) and (\ref{e26}) gives (\ref{e24}).
\end{remark}

\section{$p$-capacity to surface-area}\label{s3}
\setcounter{equation}{0}

From \cite[(5)]{P} it follows that if $n=3$ and $\Omega$ is a convex compact subset of $\mathbb R^n$ with smooth boundary $\partial\Omega$ and its mean curvature $H(\partial\Omega,\cdot)>0$ then one has the following P\'olya-Szeg\"o inequality for the electrostatic capacity and the mean radius: 
\begin{equation}
\label{e21}
\frac{\hbox{cap}_2(\Omega)}{4\pi}\le \int_{\partial\Omega}H(\partial\Omega,\cdot)\,\frac{d\sigma(\cdot)}{4\pi}
\end{equation}
with equality if $\Omega$ is a ball. This result has been extended by Freire-Schwartz to any outer-minimizing $\partial\Omega$ in $\Omega^c=\mathbb R^{n\ge 3}\setminus\Omega$, i.e., $\Omega\subseteq\Lambda\Rightarrow\hbox{area}(\partial\Omega)\le\hbox{area}(\partial\Lambda)$ (cf. \cite[Theorem 2]{FS}): 

\begin{equation}\label{e22}
\frac{\hbox{cap}_2(\Omega)}{(n-2)\sigma_{n-1}}\le\int_{\partial\Omega}
H(\partial\Omega,\cdot)\,\frac{d\sigma(\cdot)}{\sigma_{n-1}}
\end{equation}
with equality if and only if $\Omega$ is a ball. As a higher dimensional star-shaped generalization of \eqref{e21}, we have the following result whose \eqref{e29} under $p=2$ is a nice parallelism of \eqref{e22} since the outer-minimizing and the star-shaped are not mutually inclusive; see also \cite{GL}, and whose \eqref{e29e} discovers an optimal relation between the variational capacity and the surface area via the Willmore functional of the mean curvature (cf. \cite[Corollary 2]{BM} for $(p,n)=(2,3)$).

\begin{theorem}\label{t21} Let $\Omega$ be a smooth, star-shaped, compact subset of $\mathbb R^n$ with $\Omega^\circ\not=\emptyset$ and $H(\partial\Omega,\cdot)>0$. Then 

\begin{equation}\label{e29}
\frac{\hbox{cap}_p(\Omega)}{\Big(\frac{p-1}{n-p}\Big)^{1-p}\sigma_{n-1}}\le\begin{cases}\int_{\partial\Omega}\Big(H(\partial\Omega,\cdot)\Big)^{p-1}\,\frac{d\sigma(\cdot)}{\sigma_{n-1}}\ \ \hbox{as}\ \ 2\le p<n;\\
\left(\int_{\partial\Omega}\Big(H(\partial\Omega,\cdot)\Big)^{q-1}\,\frac{d\sigma(\cdot)}{\sigma_{n-1}}\right)^\frac{p-1}{q-1}\left(\frac{\hbox{area}(\partial\Omega)}{\sigma_{n-1}}\right)^\frac{q-p}{q-1}\ \hbox{as}\ 1<p\le 2\le q<n,
\end{cases}
\end{equation}
where the first inequality becomes an equality if and only if $\Omega$ is a ball. Consequently
\begin{equation}
\label{e29e}
\frac{\hbox{cap}_p(\Omega)}{\Big(\frac{p-1}{n-p}\Big)^{1-p}\sigma_{n-1}}\le \left(\frac{\hbox{area}(\partial\Omega)}{\sigma_{n-1}}\right)^\frac{n-p}{n-1}\left(\int_{\partial\Omega}\Big(H(\partial\Omega,\cdot)\Big)^{n-1}\,\frac{d\sigma(\cdot)}{\sigma_{n-1}}\right)^\frac{p-1}{n-1}\quad\forall\quad p\in (1,n)
\end{equation}
holds with equality if and only if $\Omega$ is a ball. Moreover, the limit settings $p\to 1$ or $p\to n$ in \eqref{e29e} produce 
\begin{equation}\label{e29ee}
\hbox{cap}_1(\Omega)\le\hbox{area}(\partial\Omega)\quad\hbox{or}\quad 1\le\int_{\partial\Omega}\big(H(\partial\Omega,\cdot)\big)^{n-1}\,\frac{d\sigma(\cdot)}{\sigma_{n-1}}.
\end{equation}

\end{theorem}
\begin{proof} First of all, recall that a classic solution of inverse mean curvature flow in $\mathbb R^n$ is a smooth collection $F: M^{n-1}\times [0,T)\mapsto \mathbb R^n$ of closed hypersurfaces evolving by
 \begin{equation}
 \label{eq1}
 \frac{\partial}{\partial t} F(x,t)=\frac{\tau(x,t)}{H(x,t)}\quad\forall\quad (x,t)\in M^{n-1}\times[0,T),
 \end{equation}
 where 
 $$
 H(x,t)=\hbox{div}\big(\tau(x,t)\big)>0\quad\hbox{and}\quad\tau(x,t)
 $$ 
 are the mean curvature and the outward unit normal vector of the embedded hypersurface $M_t=F(M^{n-1},t)$. According to Gerhardt \cite{Ger} (or Urbas \cite{Ur, U}), one has that for any smooth, closed, star-sharped, initial hypersurface of positive mean curvature, equation \eqref{eq1} has a unique smooth solution for all times and the rescaled hypersurfaces $M_t$ converge exponentially to a unique sphere as $t\to\infty$.
 
According to Moser's description (cf. \cite{Mos}) of the inverse mean curvature flow (whose weak formulation was studied in Huisken-Ilimanen's papers \cite{HI1, HI2}), we see that a level set formulation of the above parabolic evolution problem for hypersufaces in $\mathbb R^n$ with the initial hypersurface $M_0=\Sigma=\partial\Omega$ produces a non-negative smooth function $u$ in $\Omega^c$ such that:
 \begin{itemize}
 \item $\hbox{div}\Big(\frac{\nabla u}{|\nabla u|}\Big)=|\nabla u|$ {in} $\Omega^c;$
 
 \item $u|_{\partial\Omega}=0;$
 
 \item $u=t$ on $M_t=\Sigma_t;$
 
 \item $|\nabla u|\not=0$ in $\Omega^c$;
 
 \item $H(\Sigma_t,\cdot)=(n-1)^{-1}|\nabla u(\cdot)|$ {on} $\Sigma_t$; 
 
 \item $\hbox{area}(\Sigma_t)=e^t\hbox{area}(\partial\Omega)\ \forall\ t\ge 0.$
 
 \end{itemize}
 This function $u$ may be treated as a kind of $1$-equilibrium potential of $\Omega$ - more precisely - if $u_p=\exp\big(\frac{u}{1-p}\big)$ obeys 
 $\hbox{div}(|\nabla u_p|^{p-2}\nabla u_p)=0$ in $\Omega^c$ and $u_p|_{\partial\Omega}=1$  then $(1-p)\log u_p\to u$ locally uniformally in $\Omega^c$ as $p\to 1$; see \cite[Theorem 1.1]{Mos}. 
 
 According to (a) and the determination of $\hbox{pcap}(\cdot)$ in terms of the $(1,n)\ni p$-equilibrium potential of $\Omega$, we have
 \begin{equation}\label{e211a}
 \hbox{cap}_p(\Omega)=\hbox{cap}_p(\partial\Omega)\le\inf_{f}\int_{\mathbb R^n\setminus \Omega^\circ} |\nabla f|^p\,d\nu
 \end{equation}
 where the infimum is taken over all functions $f=\psi\circ g$ that have the above-described level hypersurfaces $(\Sigma_t)_{t\ge 0}$ and enjoy the property that $\psi$ is a one-variable function with $\psi(0)=0$ and $\psi(\infty)=1$ and $g$ is a non-negative function on $\mathbb R^n\setminus \Omega^\circ$ with $g|_{\partial\Omega}=0$ and $\lim_{|x|\to\infty}g(x)=\infty$. Note that the co-area formula yields
 \begin{equation*}\label{e211b}
 \int_{\mathbb R^n\setminus \Omega^\circ}|\nabla f|^p\,d\nu=\int_0^\infty|\psi'(t)|^p\left(\int_{\Sigma_t}|\nabla g|^{p-1}\,d\sigma_t\right)\,dt.
 \end{equation*}
 In the above and below, $d\sigma_t$ is the surface-area-element on $\Sigma_t$. So, upon choosing
$$
\begin{cases}
g=u;\\
U_p(t)=\int_{\Sigma_t}|\nabla u|^{p-1}\,\frac{d\sigma_t}{\sigma_{n-1}};\\
\psi(t)=V_p(t)=\frac{\int_0^t \big(U_p(s)\big)^{\frac{1}{1-p}}\,ds}{\int_0^\infty \big(U_p(s)\big)^{\frac{1}{1-p}}\,ds},
\end{cases}
$$
we utilize \eqref{e211a} to achieve
\begin{equation*}\label{e211}
\frac{\hbox{cap}_p(\Omega)}{\sigma_{n-1}}\le\int_0^\infty U_p(t)\Big|\frac{d}{dt}V_p(t)\Big|^p\,dt,
\end{equation*}
whence finding
\begin{equation}\label{e212}
\frac{\hbox{cap}_p(\Omega)}{\sigma_{n-1}}
\le\left(\int_0^\infty\big(U_p(t)\big)^{\frac{1}{1-p}}\,dt\right)^{1-p}.
\end{equation}

Next, let us work out the growth of $U_p(\cdot)$.

{\it Case 1: $p\in [2,n)$}. Under this assumption, utilizing \cite[Lemma 1.2, (ii)\&(v)]{HI2}, an integration-by-part, the inequality 
$$
\big(H(\Sigma_t,\cdot)\big)^2-(n-1)|\hbox{II}_t|^2\le 0
$$ 
with 
$$
0<H(\Sigma_t,\cdot)=(n-1)^{-1}|\nabla u|
$$ 
and 
$\hbox{II}_t$ being the mean curvature and the second fundamental form on $\Sigma_t$ respectively, the differentiation under the integral, we obtain
\begin{eqnarray*}
&&\frac{d}{dt}U_p(t)\\
&&=\frac{d}{dt}\left(\frac{(n-1)^{p-1}}{\sigma_{n-1}}\int_{\Sigma_t}\big(H(\Sigma_t,\cdot)\big)^{p-1}\,d\sigma_t\right)\\
&&=\frac{(n-1)^{p-1}}{\sigma_{n-1}}
\int_{\Sigma_t}\left((p-1)\big(H(\Sigma_t,\cdot)\big)^{p-2}\Big(\frac{d}{dt}H(\Sigma_t,\cdot)\Big)+\big(H(\Sigma_t,\cdot)\big)^{p-1}\right)\,d\sigma_t\\
&&=\frac{(n-1)^{p-1}}{\sigma_{n-1}}
\int_{\Sigma_t}\left(1-(p-1)\Big(\frac{|\hbox{II}_t|}{H(\Sigma_t,\cdot)}\Big)^2-(p-2)\big|\nabla \big(H(\Sigma_t,\cdot)\big)^{-1}\big|^2\right)\big(H(\Sigma_t,\cdot)\big)^{p-1}\,
d\sigma_t\\
&&\le\frac{n-p}{(n-1)\sigma_{n-1}}\int_{\Sigma_t}|\nabla u|^{p-1}\,d\sigma_t\\
&&=\Big(\frac{n-p}{n-1}\Big)U_p(t),
\end{eqnarray*}
whence discovering the following inequality through an integration
\begin{equation}\label{e213}
U_p(t)\le U_p(0)\exp\Big(t\big(\frac{n-p}{n-1}\big)\Big).
\end{equation}
Using (\ref{e212})-(\ref{e213}) we get
$$
\frac{\hbox{cap}_p(\Omega)}{\sigma_{n-1}}\le U_p(0)\left(\frac{(n-1)(p-1)}{n-p}\right)^{1-p}
$$
whence reaching the inequality in (\ref{e29}) under $2\le p<n$.

{\it Case 2: $1<p\le 2\le q<n$}. Under this situation, we use the H\"older inequality to achieve
$$
\int_{\Sigma_t}|\nabla u|^{p-1}\,\frac{d\sigma_t}{\sigma_{n-1}}\le\left( \int_{\Sigma_t}|\nabla u|^{q-1}\,\frac{d\sigma_t}{\sigma_{n-1}}\right)^\frac{p-1}{q-1}\left(\frac{\hbox{area}(\Sigma_t)}{\sigma_{n-1}}\right)^\frac{q-p}{q-1}.
$$
Now, employing the estimate for $q\in [2,n)$ and the definition of $U_p$, we obtain 
$$
U_p(t)\le \left(\int_{\partial\Omega}\big(H(\partial\Omega,\cdot)\big)^{q-1}\,\frac{d\sigma(\cdot)}{\sigma_{n-1}}\right)^\frac{p-1}{q-1}
\left(\frac{\hbox{area}(\Sigma_t)}{\sigma_{n-1}}\right)^\frac{q-p}{q-1}\exp\left(t\Big(\frac{n-p}{n-1}\Big)\right).
$$
Bringing this last inequality into \eqref{e212}, along with 
$$
\hbox{area}(\Sigma_t)=e^t\hbox{area}(\partial\Omega),
$$
we arrive at the second inequality of \eqref{e29}. 

{\it Case 3: equality of \eqref{e29}}. If $\Omega$ is a ball, then a direct computation gives equality of \eqref{e29}. Conversely, if the inequality $\le$ in \eqref{e29} becomes an equality, then the above-established differential inequalities for $U_p$ force
$$
\big(H(\Sigma_t,\cdot)\big)^2-(n-1)|\hbox{II}_t|^2=0\quad\hbox{on}\quad\Sigma_t,
$$ 
which in turn ensures that $\Sigma_t$ consists of the union of disjoint spheres. Since $\Sigma_t$ is generated by a smooth solution of the inverse mean curvature flow in $\mathbb R^n$, $\Sigma_t$ must be a single sphere. Consequently, $\Omega$ is a ball.

After that, \eqref{e29e} and its equality case follow from \eqref{e29} and its equality case as well as the following estimate (based on the H\"older inequality)
$$
\int_{\partial\Omega}\big(H(\partial\Omega,\cdot)\big)^{q-1}\,\frac{d\sigma(\cdot)}{\sigma_{n-1}}\le\left(\int_{\partial\Omega}\big(H(\partial\Omega,\cdot)\big)^{n-1}\,\frac{d\sigma(\cdot)}{\sigma_{n-1}}\right)^\frac{q-1}{n-1}\left(\frac{\hbox{area}(\partial\Omega)}{\sigma_{n-1}}\right)^\frac{n-q}{n-1}\quad\forall\quad q\in (1,n).
$$

Finally, let us check \eqref{e29ee}. On the one hand, letting $p\to 1$ in \eqref{e29e} yields the Mazya inequality (cf. \cite[p.149, Lemma 2.2.5]{Maz}):
$$
cap_1(\Omega)\le \hbox{area}(\partial\Omega).
$$
On the other hand, choosing $0<r<R$ with
$B(x_0,r)\subseteq \Omega\subseteq B(x_0,R)$, we utilize the properties (b)\&(d) of $\hbox{cap}_p(\cdot)$ to derive
$$
r^{n-p}\le \frac{\hbox{cap}_p(\Omega)}{\Big(\frac{p-1}{n-p}\Big)^{1-p}\sigma_{n-1}}\le R^{n-p},
$$
whence achieving
$$
\lim_{p\to n}\frac{\hbox{cap}_p(\Omega)}{\Big(\frac{p-1}{n-p}\Big)^{1-p}\sigma_{n-1}}=1.
$$
This, together with letting $p\to n$ in \eqref{e29e}, derives the Willmore inequality (cf. \cite[p. 87]{RS} or \cite{chen} for immersed hypersurfaces in $\mathbb R^n$):
$$
1\le\int_{\partial\Omega}\big(H(\partial\Omega,\cdot)\big)^{n-1}\,\frac{d\sigma(\cdot)}{\sigma_{n-1}}.
$$
\end{proof}

\begin{remark}
\label{r1} Two comments are in order:

{\rm (i)} Let $\Omega$ be a smooth compact subset of $\mathbb R^n$ with $\Omega^\circ\not=\emptyset$ and $H(\partial\Omega,\cdot)>0$. If $\partial\Omega$ is outer-minimizing, then one has the $(1,n)\ni p$-Aleksandrov-Fenchel inequality:
\begin{equation}\label{e210}
\left(\frac{\hbox{area}(\partial\Omega)}{\sigma_{n-1}}\right)^\frac{n-p}{n-1}\le\begin{cases}
\int_{\partial\Omega}\big(H(\partial\Omega,\cdot)\big)^{p-1}\,\frac{d\sigma(\cdot)}{\sigma_{n-1}}\ \ \hbox{as}\ \ 2\le p<n;\\
{\left(\int_{\partial\Omega}\Big(H(\partial\Omega,\cdot)\Big)^{q-1}\,\frac{d\sigma(\cdot)}{\sigma_{n-1}}\right)^\frac{p-1}{q-1}}{\left(\frac{\hbox{area}(\partial\Omega)}{\sigma_{n-1}}\right)^\frac{q-p}{q-1}}\ \ \hbox{as}\ \ 1<p\le 2\le q<n,
\end{cases}
\end{equation}
where the first inequality becomes an equality if and only if $\Omega$ is a ball.

In fact, using the known $2$-Aleksandrov-Fenchel inequality (cf. \cite[Theorem 2(b)]{FS})
\begin{equation}
\label{AF}
\left(\frac{\hbox{area}(\partial\Omega)}{\sigma_{n-1}}\right)^\frac{n-2}{n-1}\le
\int_{\partial\Omega}H(\partial\Omega,\cdot)\,\frac{d\sigma(\cdot)}{\sigma_{n-1}}
\end{equation}
and the H\"older inequality, we gain
$$
\int_{\partial\Omega}H(\partial\Omega,\cdot)\,\frac{d\sigma(\cdot)}{\sigma_{n-1}}\le \left(\int_{\partial\Omega}\big(H(\partial\Omega,\cdot)\big)^{p-1}\,\frac{d\sigma(\cdot)}{\sigma_{n-1}}\right)^\frac1{p-1}\left(\frac{\hbox{area}(\partial\Omega)}{\sigma_{n-1}}\right)^{\frac{p-2}{p-1}}\quad\forall\quad p\in [2,n),
$$
whence implying \eqref{e210}. If the first inequality of \eqref{e210} becomes equality, then equality of \eqref{AF} is valid, and hence $\Omega$ is a ball. Of course, the converse follows from a direct computation.

{\rm (ii)} An application of \eqref{e22}, \eqref{e210} and the H\"older inequality derives that if $\Omega\subset\mathbb R^{n\ge 3}$ is a smooth compact set with $\Omega^\circ\not=\emptyset$ and $\partial\Omega$ being outer-minimizing as well as having $H(\partial\Omega,\cdot)>0$ then one has the following log-convexity type inequality for the electrostatic capacity, the surface area and the Willmore functional:
\begin{equation}\label{e20aa}
{\frac{\hbox{cap}_2(\Omega{})}{(n-2)\sigma_{n-1}}}\le{\left(\frac{\hbox{area}(\partial\Omega)}{\sigma_{n-1}}\right)^\frac{n-2}{n-1}}
\left(\int_{\partial\Omega}\big(H(\partial\Omega,\cdot)\big)^{n-1}\,\frac{d\sigma(\cdot)}{\sigma_{n-1}}\right)^\frac1{n-1}
\end{equation}
with equality if and only if $\Omega$ is a ball. Interestingly and naturally, \eqref{e20aa} and \eqref{e29e} under $p=2$ complement each other thanks to the relative independence between the outer-minimizing and the star-shaped.  
\end{remark}

\end{document}